\documentclass[reqno]{amsart}

\usepackage{amsfonts,color,newclude}
\usepackage{amsthm}
\usepackage{amssymb,transparent}
\usepackage{amsmath}
\usepackage{amscd}
\usepackage{stmaryrd}
\usepackage{xypic}
\input xy
\xyoption{all}

\usepackage{mathabx}
\usepackage{graphicx}
\usepackage{subfiles}
\usepackage{tikz}
\usepackage{MnSymbol}
\usepackage[colorlinks=true, pdfstartview=FitV, linkcolor=blue, citecolor=blue, urlcolor=blue]{hyperref}
\usepackage{hyperref}
\usepackage{thm-restate}

\renewcommand{\tilde}{\widetilde}
\newcommand{\mathsym}[1]{{}}

\newcommand{\co}{\colon\thinspace}

\newcommand{\inv}{^{-1}}

\renewcommand{\bar}{\overline}

\newcommand{\Stab}{\operatorname{Stab }}
\newcommand{\Hull}{\operatorname{Hull}}
\newcommand{\cal}[1]{\mathcal{#1}}

\def\mc {\mathcal}

\newcommand{\CAT} {\ensuremath{\operatorname{CAT}}}

\newcommand{\orbit}[1] {\llbracket {#1} \rrbracket}

\newcommand{\rightQ}[2]{\left.\raisebox{.2em}{$#1$}\middle/\raisebox{-.2em}{$#2$}\right.}
\newcommand{\leftQ}[2]{\left.\raisebox{-.2em}{$#2$}\middle\backslash\raisebox{.2em}{$#1$}\right.}
\newtheorem{theorem}{Theorem}[section]
\newtheorem{definition}[theorem]{Definition}

\newtheorem{proposition}[theorem]{Proposition}
\newtheorem{cor}[theorem]{Corollary}
\newtheorem{lemma}[theorem]{Lemma}

\newtheorem{remark}[theorem]{Remark}

\newtheorem{notation}[theorem]{Notation}
\newtheorem{assumption}[theorem]{Assumption}
\newtheorem{convention}[theorem]{Convention}

\title{Relatively geometric actions on \CAT$(0)$ cube complexes}

\author{Eduard Einstein}
\author{Daniel Groves}
\address{Department of Mathematics, Statistics, and Computer Science,
University of Illinois at Chicago,
322 Science and Engineering Offices (M/C 249),
851 S. Morgan St.,
Chicago, IL 60607-7045}
\email{ede17@pitt.edu}
\email{groves@math.uic.edu}

\thanks{The first author is supported in part by an AMS-Simons Travel Grant.  The second author is supported in part by NSF grant DMS-1904913. AMS Subject classification (2020): 20F67, 20F65, 57M07.}

\begin{document}
\begin{abstract}
We develop the foundations of the theory of relatively geometric actions of relatively hyperbolic groups on \CAT$(0)$ cube complexes, a notion introduced in our previous work \cite{RelCannon}.  In the relatively geometric setting we prove: full relatively quasi-convex subgroups are convex compact; an analog of Agol's Theorem; and a version of Haglund--Wise's Canonical Completion and Retraction.  
\end{abstract}

\maketitle

\section{Introduction}
The interaction between the geometry of \CAT$(0)$ cube complexes and that of hyperbolic groups is at the center of some of the most powerful aspects of Haglund and Wise's theory of (virtually) special cube complexes \cite{HW08}, which in turn was at the center of the resolution of the Virtual Haken Conjecture \cite{AgolVirtualHaken}, and the Virtual Fibering Conjecture \cite{AgolVirtualHaken} in the closed case.  In the case of finite-volume hyperbolic $3$--manifolds, the Virtual Fibering Conjecture was resolved by Wise \cite{WiseManuscript}, also using virtually special cube complexes, but now using the relatively hyperbolic geometry of the fundamental group.

In search of more general results, there have been numerous papers dealing with relatively hyperbolic groups acting on \CAT$(0)$ cube complexes.  See, for example, \cite{HruskaWise14}, \cite{SW2015}, or \cite{PW}.  These papers typically deal with proper actions, which are either co-compact, or co-sparse.

In \cite{RelCannon}, we introduced a new kind of action for a relatively hyperbolic group on a \CAT$(0)$ cube complex, a \emph{relatively geometric action} (see Definition~\ref{def:RG} below).  If a relatively hyperbolic pair $(G,\mc{P})$ acts relatively geometrically on a space $X$ then $X$ is quasi-isometric to the coned-off Cayley graph for $(G,\mc{P})$, and hence in particular is $\delta$--hyperbolic for some $\delta$.  We believe that relatively geometric actions on \CAT$(0)$ cube complexes should have as rich a theory as hyperbolic groups acting properly and cocompactly on \CAT$(0)$ cube complexes.  In this paper, we begin a systematic investigation of such actions.  In a relatively geometric action, the parabolic subgroups $\mc{P}$ act elliptically on the cube complex, so in contrast to either of the kinds of proper actions mentioned above, relatively geometric actions may exist even when the parabolic subgroups do not act on \CAT$(0)$ cube complexes in interesting ways.  Therefore, we expect that the class of relatively hyperbolic groups acting relatively geometrically on \CAT$(0)$ cube complexes is substantially broader than the class acting geometrically, thus bringing these powerful techniques to bear in a much wider setting.  A final reason that relatively geometric actions are desirable is that the techniques of Dehn filling developed by the second author and Manning in \cite{GMOmnibus} are applicable to relatively geometric actions, as we used already in \cite{RelCannon}.  This is a theme we further explore in this paper.

We now state our main results.  First, recall the definition of relatively geometric (see Section~\ref{s:prelim} for other definitions).

\begin{definition} \label{def:RG}
 Suppose that $(G,\cal{P})$ is a group pair.  A (cellular) action of $G$ on a cell complex $\widetilde{X}$ is {\em relatively geometric (with respect to $\cal{P}$)} if 
\begin{enumerate}
\item  $\leftQ{\widetilde{X}}{G}$ is compact;
\item Each element of $\cal{P}$ acts elliptically on $\widetilde{X}$; and
\item Each stabilizer in $G$ of a cell in $\widetilde{X}$ is either finite or else conjugate to a finite-index subgroup of an element of $\cal{P}$.
\end{enumerate}
\end{definition}

In order to state our results,
we make the following standing assumption.
\begin{assumption} \label{as:RG}
Suppose that $(G,\mc{P})$ is relatively hyperbolic and that $\widetilde{X}$ is a \CAT$(0)$ cube complex which admits a relatively geometric action of $G$ with respect to $\mc{P}$.  Let $X = \leftQ{\widetilde{X}}{G}$.
\end{assumption}

In Sageev--Wise \cite{SW2015} and Haglund \cite{Haglund2008} it is proved that if a hyperbolic group $G$ acts geometrically on a \CAT$(0)$ cube complex and $H$ is a quasi-convex subgroup of $G$ then there is a convex $H$--invariant and $H$--cocompact subcomplex.  Thus, $H$ also acts properly and cocompactly on a \CAT$(0)$ cube complex, and in the virtually special setting this is what allows the canonical completion and retraction construction of Haglund--Wise \cite{HW08} to be applied.

Our first result, proved in Section~\ref{s:Basic props}, is a relatively geometric analogue of the above-mentioned results of \cite{SW2015} and \cite{Haglund2008}.  For a relatively hyperbolic analogue in the proper and cocompact or cosparse settings, see \cite{SW2015}.

\begin{restatable}{theorem}{CCC} \label{t:CCC}
Under Assumption~\ref{as:RG}, for any full relatively quasi-convex subgroup $H$ of $G$ and any compact $K \subset \widetilde{X}$ there exists a convex $H$--invariant sub-complex $\widetilde{Y}$ of $\widetilde{X}$ so that $K \subset \widetilde{Y}$ and  $\leftQ{\widetilde{Y}}{H}$ is compact.
\end{restatable}

In order to solve the Virtual Haken and Virtual Fibering Conjectures, Agol proved that any hyperbolic group acting geometrically on a \CAT$(0)$ cube complex is virtually special.  The next result is a relatively geometric version of Agol's Theorem, and is proved in Section~\ref{s:Dehn}.

\begin{restatable}{theorem}{RGAGol} \label{t:RGA}
Under Assumption~\ref{as:RG}, if the elements of $\mc{P}$ are residually finite then there is a finite-index subgroup $G_0$ of $G$ so that $\leftQ{\widetilde{X}}{G_0}$ is a special cube complex.
\end{restatable}

Note that $\leftQ{\widetilde{X}}{G_0}$ should be considered as a complex of groups, rather than merely as a space, so the underlying space being a special cube complex is not obviously as useful as in the case of Agol's Theorem.   However, there are still favorable separability properties -- see Theorem~\ref{t:VVR} and Corollary~\ref{c:Sep} below.

In Section \ref{s:Dehn}, we consider a full relatively quasi-convex subgroup $H$ of $G$ and explore the behavior of the hull $\widetilde{Y}$  found in Theorem~\ref{t:CCC} under long Dehn fillings.  In particular, a special case of what we prove is the following result, proved in Section~\ref{s:Dehn}.  See Sections~\ref{s:prelim} and \ref{s:Dehn} for definitions of terms, and Proposition~\ref{prop:fill CAT(0)} for a description of the subgroups $\mc{Q}$.

\begin{restatable}{theorem}{FillHull} \label{t:fill hull}
Make Assumption~\ref{as:RG}, and suppose that elements of $\mc{P}$ are residually finite.  Let $H \le G$ be a full relatively quasi-convex subgroup and let $\widetilde{Y}$ be as in Theorem~\ref{t:CCC}.  For sufficiently long $\left( \mc{Q} \cup H \right)$--fillings $G \to G/K$, if $K_H = K \cap H$ then $\overline{X} = \leftQ{\widetilde{X}}{K}$ and $\overline{Y} = \leftQ{\widetilde{Y}}{K_H}$ are both \CAT$(0)$ cube complexes and the natural map $\overline{Y} \to \overline{X}$ is an embedding with convex image.
\end{restatable}

In Section \ref{s:CCR}, we develop a relatively geometric version of the Canonical Completion and Retraction (see Theorem~\ref{t:CCR}).  An application of this construction is the following, proved in Section~\ref{s:CCR}.
\begin{restatable}{theorem}{VVR} \label{t:VVR}
Under Assumption~\ref{as:RG}, if the elements of $\mc{P}$ are residually finite then for any full relatively quasi-convex subgroup $H$ of $G$ there is a finite-index subgroup $H_0 \le H$ which is a retract of a finite-index subgroup of $G$.  
\end{restatable}

\begin{cor}\label{c:Sep}
Under Assumption~\ref{as:RG}, if every $P\in\mc{P}$ is residually finite:
\begin{enumerate}
\item $G$ is residually finite, and
\item every full relatively quasiconvex subgroup of $G$ is separable. 
\end{enumerate}
\end{cor}
The first item is an immediate consequence of Theorem~\ref{t:RHDF} and Corollary~\ref{cor:fill to get VS hyp}. The second item follows from the residual finiteness of $G$ and Theorem~\ref{t:VVR}.   The separability properties of Corollary~\ref{c:Sep} also follow from \cite[Theorem 4.7]{GM20}, which works in the more general setting of ``weakly relatively geometric" actions.

We expect the tools in this paper to be of foundational interest for the further study of relatively geometric actions.  In a future work, we plan to prove a relatively geometric version of the cubulation result of Hsu--Wise \cite{HW2015}. We expect such a relatively geometric analogue of the Hsu--Wise results to be the key to proving a relatively geometric version of Wise's Quasi-convex Hierarchy Theorem, one of the main results from \cite{WiseManuscript}.  Such a result would provide a wealth of examples of relatively geometric actions on \CAT$(0)$ cube complexes.

\begin{convention}
Conjugation of $g$ by $h$, written $g^h$, is $hgh^{-1}$.
\end{convention}

\thanks{The authors thank the referee for a very careful and thoughtful report, which improved this paper.}

\section{Preliminaries}\label{s:prelim}

In this section we collect some basic definitions and results about relatively hyperbolic groups and relatively hyperbolic Dehn filling needed for this paper.

\subsection{Relatively hyperbolic groups and relatively quasi-convex subgroups}

\begin{definition}[Combinatorial horoball] \cite[Definition~3.1]{GM08}
Let $\Gamma$ be a $1$--complex.  The \emph{combinatorial horoball based on $\Gamma$}, denoted $\mc{H}(\Gamma)$, is the $1$-complex whose vertices are $\Gamma^{(0)} \times \left( \{ 0 \} \cup \mathbb N \right)$, and whose edge set consists of (i) edges between $(v,0)$ and $(w,0)$ whenever there is an edge between $v$ and $w$ in $\Gamma$; (ii) for all $k > 0$ and all $v,w \in \Gamma^{(0)}$ so that $0 < d_\Gamma(v,w) \le 2^k$, an edge between $(v,k)$ and $(w,k)$; and (iii) edges between $(v,k)$ and $(v,k+1)$ for all $v \in \Gamma^{(0)}$ and all $k \ge 0$.
\end{definition}
In \cite{GM08}, $2$--cells are added to combinatorial horoballs, but we do not need them here.  According to \cite[Theorem~3.8, Remark~3.9]{GM08}, $\mc{H}(\Gamma)$ is $20$--hyperbolic for any connected $1$--complex $\Gamma$.

\begin{definition}[Cusped space] \cite[Definition~3.12]{GM08} \label{def:cusped space}
Let $G$ be a group and $\mc{P}$ a finite collection of subgroups of $G$.  Further, let $S$ be a generating set for $G$ so that $\left\langle S \cap P \right\rangle = P$ for each $P \in \mc{P}$.  Let $\Gamma(G,S)$ be the Cayley graph of $G$ with respect to $S$.  The \emph{cusped space} for $(G,\mc{P})$, denoted $\mc{C}(G,\mc{P},S)$, is obtained by gluing a copy of the combinatorial horoball over the Cayley graph $\Gamma_P(S\cap P)$ of $P$ (with respect to $S \cap P$) to each left translate of the natural copy of $\Gamma_P(S \cap P)$ in $\Gamma(G,S)$.
\end{definition}

\begin{definition} \cite[Theorem~3.25]{GM08} \label{d:RH}
Let $(G,\mc{P})$ be a group pair, with $G$ finitely generated, $\mc{P}$ finite, and each element of $\mc{P}$ finitely generated.  Then $(G,\mc{P})$ is relatively hyperbolic if for some (any) finite generating set $S$ for $G$ so that $\left\langle S \cap P \right\rangle = P$ for each $P \in \mc{P}$, the cusped space $\mc{C}(G,\mc{P},S)$ is $\nu$--hyperbolic, for some $\nu$.
\end{definition}
There are analogous definitions of relatively hyperbolic pairs $(G,\mc{P})$ when $G$ is not finitely generated, or when $\mc{P}$ is not finite (see \cite{Hruska2010,OsinRelHyp}, for example) but we do not need them here.  The hyperbolicity of the cusped space for $(G,\mc{P)}$ does not depend on the choice of $S$, though of course the value of $\nu$ does.

\begin{convention}
Whenever $(G,\mc{P})$ is relatively hyperbolic, we assume elements of $\mc{P}$ are infinite.  If this is not the case, discard the finite elements.  This does not affect the relative hyperbolicity, or whether or not a given action is relatively geometric.
\end{convention}

\begin{definition}
Let $(G,\cal{P})$ be relatively hyperbolic, and let $H \le G$.  The {\em induced peripheral structure on $H$} is a collection $\cal{D}$ of representatives of $H$--conjugacy classes of maximal infinite parabolic subgroups of $H$.  
\end{definition}
Suppose $(G,\cal{P})$ is relatively hyperbolic, that $H \le G$, and let $\cal{D}$ be the induced peripheral structure on $H$.
Given $D \in \cal{D}$ there exists $P_D \in \cal{P}$ and $c_D \in G$ so that $D \le P_D^{c_D}$.  
Associated to the pairs $(G,\mc{P})$ and $(H,\mc{D})$ (along with choices of generating sets) are cusped spaces $X_G$ and $X_H$.  As explained in \cite[$\S3$]{AGM09}, the inclusion $\iota \co H \to G$ extends to an $H$--equivariant Lipschitz proper map $\check{\iota} \co X_H \to X_G$.

\begin{definition} \cite[Definition 2.9]{GM-QCDF} (see also \cite[Definiton 3.11]{AGM09})
Suppose $(G,\mc{P})$ is a relatively hyperbolic group, that $H \le G$ and that $\mc{D}$ is the induced peripheral structure on $H$.  The subgroup $H$ is \emph{relatively quasi-convex} in $(G,\mc{P})$ if $\check{\iota}(X_H)$ is quasi-convex in $X_G$.
\end{definition}
There are many equivalent characterizations of quasi-convexity in relatively hyperbolic groups.  See Hruska \cite{Hruska2010} for five other conditions (which Hruska proves are all equivalent), and see \cite{MMP} and \cite{GM-QCDF} for a proof that the above definition is equivalent to Hruska's.

\begin{definition}
Let $(G,\mc{P})$ be a group pair, $H \le G$ a subgroup, and $\mc{D}$ the induced peripheral structure on $H$.  The subgroup $H$ is \emph{full} if for every $D \in \mc{D}$, $D$ is finite-index in a maximal parabolic.
\end{definition}

Since we assume elements of $\mc{P}$ are finitely generated, and relatively quasi-convex subgroups are themselves relatively hyperbolic we have the following immediate consequence.

\begin{lemma} \label{l:fRQC fg}
Any full relatively quasi-convex subgroup of $(G,\mc{P})$ is finitely generated.
\end{lemma}

\begin{definition}
Let $(G,\mc{P})$ be relatively hyperbolic and $\mc{C}$ be the cusped space as in Definition~\ref{def:cusped space}.

A geodesic $\gamma$ in $\mc{C}$ \emph{penetrates a horoball $B$ to depth $R>0$} if there exists a $\xi\in \gamma\cap B$ with $d(\xi,\mc{C}\setminus B)\ge R$.
\end{definition}

\begin{proposition}\cite[Proposition~A.6]{MMP} \label{prop:MMP}
Let $(G,\mc{P})$ be relatively hyperbolic, $\mc{C}$ a cusped space for $(G,\mc{P})$, and $H \le G$ relatively quasi-convex.  There exists a constant $R$ depending on $G$, $\mc{C}$ and $H$ so that whenever  a horoball $B$ in $\mc{C}$ is $R$--penetrated by $H$ we have $| \Stab_G(B) \cap H | = \infty$.
\end{proposition}

\subsection{Relatively hyperbolic Dehn filling}
 Suppose $(G,\cal{P})$ is a group pair and that $\cal{N} = \{ N_P \unlhd P \mid P \in \cal{P} \}$ is a collection of normal subgroups of the peripheral groups.  The {\em Dehn filling of $(G,\cal{P})$ determined by $\cal{N}$} is the pair $(\overline{G},\overline{\cal{P}})$, where $\overline{G} = G/K$, for $K$ the normal closure in $G$ of $\bigcup \cal{N}$, and $\overline{\cal{P}}$ is the image of $\cal{P}$ in $\overline{G}$.  The elements of $\cal{N}$ are called {\em filling kernels}.  We sometimes also write $\overline{G} = G\left( N_P \mid P \in \cal{P} \right)$, when we want to make the dependence on the choice of filling kernels explicit.
 
The filling is {\em peripherally finite} if $N_P$ is finite-index in $P$ for all $P \in \cal{P}$.  If $H < G$ then the filling is an \emph{$H$--filling} if for every $g \in G$, $|H \cap P^g| = \infty$ implies $N_P^g \le H$.  If $\cal{H}$ is a family of subgroups, an \emph{$\cal{H}$--filling} is a filling which is an $H$--filling for all $H \in \cal{H}$.

A property $\mathsf{P}$ holds {\em for all sufficiently long fillings of $(G,\cal{P})$} if there is a finite set $S \subset \bigcup \cal{P}$ so that $\mathsf{P}$ holds for any filling where $\left( \bigcup \cal{N} \right) \cap S = \emptyset$.  There is an obvious meaning to phrases such as `for all sufficiently long $H$--fillings', etc.

The following result is the basic result of relatively hyperbolic Dehn filling, and is due to Osin \cite{osin:peripheral} (see also \cite{GM08}).

\begin{theorem} \cite[Theorem 1.1]{osin:peripheral} \label{t:RHDF}
Suppose $(G,\mc{P})$ is relatively hyperbolic, and $\mc{F} \subset G$ is finite.  For sufficiently long Dehn fillings $G \to \overline{G} = G\left( N_P \mid P \in \mc{P} \right)$ we have:
\begin{enumerate}
\item For each $P \in \mc{P}$, the canonical map $P/N_P \to \overline{G}$ is injective.  Denote the image by $\overline{P}$;
\item $\left( \overline{G}, \left\{ \overline{P} \mid P \in \mc{P} \right\} \right)$ is relatively hyperbolic;
\item The map $G \to \overline{G}$ is injective on $\mc{F}$.
\end{enumerate}
\end{theorem}

In \cite{AGM09,AgolVirtualHaken,GM-QCDF} various results have been proved controlling the image of relatively quasi-convex subgroups under Dehn filling.  We follow \cite{GM-QCDF}.

\begin{definition} \label{def:induced filling}
Let $(G,\mc{P})$ be relatively hyperbolic, and $H \le G$ be relatively quasi-convex.  Let $\mc{D}$ be the induced peripheral structure on $H$, and for each $D \in \mc{D}$ let $P_D \in \mc{P}$ and $c_D \in G$ be so $D \le P_D^{c_D}$.  Let $G \to G\left( N_P \mid P \in \mc{P} \right)$ be a Dehn filling.  For $D \in \mc{D}$, write $N_D = N_{P_D}^{c_D} \cap D$.  The \emph{induced filling kernels for $(H,\mc{D})$} are $\left\{ N_D \right\}_{D \in \mc{D}}$.   The \emph{induced filling for $H$} is the filling $H \to H\left( N_D \mid D \in \mc{D} \right)$
\end{definition}
Given $(G,\mc{P})$, $H \le G$, a Dehn filling $G \to G \left( N_P \mid P \in \mc{P}\right)$, and the induced filling $H \to H \left( N_D \mid D \in \mc{D} \right)$ as in Definition~\ref{def:induced filling} it is clear that there is an induced map $\overline{\iota} \co H \left( N_D \mid D \in \mc{D} \right) \to 
G \left( N_P \mid P \in \mc{P}\right)$.

The following is an immediate consequence of \cite[Propositions~4.5 and 4.6]{GM-QCDF} and \cite[Lemma 3.7]{GM-QCDF} (which states that a sufficiently long $H$--filling is sufficiently ``$H$--wide").
\begin{theorem} \label{t:Dehn QC}
Suppose $(G,\mc{P})$ is relatively hyperbolic and $H \le G$ is relatively quasi-convex.  For sufficiently long $H$--fillings $G \to \overline{G}$, the
induced map $\overline{\iota} \co H \left( N_D \mid D \in \mc{D} \right) \to 
G \left( N_P \mid P \in \mc{P}\right)$ is injective, with relatively quasi-convex image.
\end{theorem}

\section{Basic properties of relatively geometric actions}\label{s:Basic props}

Throughout this section, we make Assumption~\ref{as:RG}.
Let $q \co \widetilde{X} \to \leftQ{\widetilde{X}}{G}$ be the quotient map.

As noted in \cite{RelCannon}, the results of Charney--Crisp \cite[Theorem 5.1]{CharneyCrisp} immediately imply the following result.

\begin{proposition} \label{prop:Xtilde hyp}
The space $\widetilde{X}$ is quasi-isometric to the coned-off Cayley graph of $(G,\mc{P})$, and in particular is $\delta$--hyperbolic for some $\delta$.
\end{proposition}
Using one of Hruska's characterizations of relatively quasi-convexity from \cite{Hruska2010} we obtain the following result.
\begin{cor}\label{cor:RQC QC}
Suppose that $H$ is a relatively quasi-convex subgroup of $G$, and that $x \in \widetilde{X}$.  The orbit $H \cdot x$ is quasi-convex in $\widetilde{X}$.
\end{cor}
\begin{proof}
According to \cite[(QC-5)]{Hruska2010}, any $H$--orbit in the coned-off Cayley graph is quasi-convex (in fact, (QC-5) is stronger than this).  Since $\widetilde{X}$ is $\delta$--hyperbolic, quasi-isometries take quasi-convex subsets to quasi-convex subsets, so a quasi-isometry of pairs $\left( G, H \right) \to \left(\widetilde{X},H \cdot x\right)$ implies the result. 
\end{proof}

\subsection{Stabilizers of cells}

\begin{notation}
Suppose $G$ is a group and that $H_1, H_2 \le G$.  If $H_1$ and $H_2$ are commensurable in $G$ then we write $H_1 \sim H_2$.
\end{notation}

\begin{definition} \label{def:para stab}
Let $(G,\cal{P})$ be relatively hyperbolic and let $G$ admit a relatively geometric action on a \CAT$(0)$ cube complex $\widetilde{X}$.
For $P \in \mc{P}$ define a sub-complex of $\tilde{X}$
\[	C(P) = \bigcup\limits_{\mbox{\tiny{cells }} \sigma} \{ \sigma \mid \Stab(\sigma) \sim P	\}	.	\]
\end{definition}

\begin{proposition} \label{p:C compact}
For any $P \in \mc{P}$ the set $C(P)$ is compact and convex.
\end{proposition}
\begin{proof}
Recall that elements of $\mc{P}$ are almost malnormal (see, for example, Farb \cite[Example 1, p.819]{FarbRelHypGroups}).  Suppose $\sigma \subseteq C(P)$ is a cell and $g \in G \smallsetminus P$.  The stabilizer of $g \cdot \sigma$ is $\Stab(\sigma)^g$, which intersects $P$ in a finite group (since $\Stab(\sigma) \sim P$), and hence $g \cdot \sigma \nsubseteq C(P)$.

It follows that the quotient map $q|_{C(P)} \co C(P) \to X$ is finite-to-one, and since $X$ is compact, $C(P)$ is compact.  It is clear that if $x,y \in C(P)$ then $\Stab(x) \sim P \sim \Stab(y)$ and that if $\gamma$ is the geodesic between $x$ and $y$ then $\Stab(x) \cap \Stab(y) \sim P$ and fixes $\gamma$ pointwise. It follows that $\gamma \subseteq C(P)$, so $C(P)$ is convex, as required.
\end{proof}

We now recall a definition and result from Haglund \cite{Haglund2008}.
\begin{definition} \cite[Definition 2.14]{Haglund2008}
Let $\widetilde{Z}$ be a \CAT$(0)$ cube complex and $A$ be a sub-complex of $\widetilde{Z}$.  The (combinatorial) convex hull of $A$, denoted $\Hull(A)$, is the intersection of all convex sub-complexes containing $A$.
\end{definition}
Given a \CAT$(0)$ cube complex $\widetilde{Z}$ and a hyperplane $H$ in $\widetilde{Z}$, denote $\widetilde{Z} \bbslash H$ the union of cubes of $\widetilde{Z}$ whose intersection with $H$ is empty.  This has two connected components, which are called (combinatorial) half-spaces.  See \cite[Definition 2.15]{Haglund2008} for more details.  According to \cite[Proposition 2.17]{Haglund2008}, any convex sub-complex $Y$ of $\widetilde{Z}$ is the intersection of the half-spaces containing $Y$.  Hence, for any sub-complex $A$ of $\widetilde{Z}$, $\Hull(A)$ is the intersection of the half-spaces containing $A$.

\subsection{Cocompact Cores for Relatively Geometric Groups}

We continue to make Assumption~\ref{as:RG}. 
In this section, we prove Theorem~\ref{t:CCC}. This result builds on \cite{SW2015}.  By Proposition~\ref{prop:Xtilde hyp}, $\widetilde{X}$ is $\delta$--hyperbolic, and by Corollary~\ref{cor:RQC QC}, $H$ acts ``quasi-convexly" on $\widetilde{X}$, in the terminology of \cite{SW2015}.  Thus, the existence of a convex subset of $\widetilde{X}$ preserved by $H$ and finite distance from any given $H$--orbit follows immediately from \cite[Proposition 3.3]{SW2015}. However, $\widetilde{X}$ is locally infinite, so it still takes work to prove that the action of $H$ on this core is cocompact. 
Recall the statement of Theorem~\ref{t:CCC}.

\CCC*
\begin{proof}
The subgroup $H$ is finitely generated by Lemma~\ref{l:fRQC fg}.  Let $h_1,\ldots,h_k$ generate $H$.  By replacing $K$ with the union of the cells whose interiors intersect $K$ we may assume $K$ is a finite sub-complex, and we may add finitely many $1$--cells so $K$ is connected. Since $\Hull(K) = \Hull(K^{(1)})$ we may replace $K$ by its $1$--skeleton.  Enlarge $H \cdot K$ to a connected subset of $\widetilde{X}^{(1)}$ by fixing $x \in K^{(0)}$ and choosing, for each $i \in \{ 1 ,\ldots , k \}$, a geodesic $\gamma_{i}$ between $x$ and $h_i \cdot x$, and replacing $H\cdot K$ by $H \cdot K \cup \left\{ H \cdot \gamma_{i} \mid 1 \le i \le k \right\}$.  Denote this (connected, $H$--cocompact) sub-graph of $\widetilde{X}^{(1)}$ by $\Gamma_H$.

It suffices to prove the combinatorial convex hull of $\Gamma_H$ is $H$--cocompact.  By Corollary~\ref{cor:RQC QC}, $H$--orbits in $\widetilde{X}$ are quasi-convex, so by \cite[Proposition 3.3]{SW2015} there exists $D\ge 0$ so that $\operatorname{Hull}(\Gamma_H)\subseteq \cal{N}_D(H \cdot K)$.
For the following, let $Z:= \Hull(\Gamma_H)$, let $z\in Z^{(0)}$ and let $A = \Stab_H(z)$. 
Let $\mc{E}_{z,H}$ denote the collection of edges $e$ adjacent to $z$ in $\widetilde{X}^{(1)}$ for which the hyperplane $W_e$ dual to $e$ intersects $\Gamma_H$ nontrivially.  Observe $\mc{E}_{z,H}$ is precisely the set of edges adjacent to $z$ which lie in $Z$.  We claim there are only finitely many $A$--orbits of edges in $\mc{E}_{z,H}$.  If $\Stab_G(z)$ is finite, then this is clear from the fact that $\leftQ{\widetilde{X}}{G}$ is compact.  On the other hand, if $A$ is infinite, then since $H$ is full and the $G$--action on $\widetilde{X}$ is relatively geometric, $A$ acts cocompactly on the edges adjacent to $z$, so again the claim is clear.  If $z$ has finite valence there is nothing to show.

The remaining case is that $\Stab_G(z)$ is infinite, $z$ has infinite valence, and $A$ is finite.  For simplicity, since we have fixed $H$ and $z$, we simply write $\mc{E}$ for $\mc{E}_{z,H}$.

Towards a contradiction, suppose $\mc{E}$ is infinite.  For $e \in \mc{E}$ with dual hyperplane $W_e$, let $x_e \in \Gamma_H \cap W_e$, and let $Q_e = \Stab_G(W_e)$.
By passing to an infinite subset of $\mc{E}$ we may assume:

\begin{enumerate}
\item All elements of $\mc{E}$ belong to the same $\Stab_G(z)$--orbit, and
\item All $x_e$ lie in the same $H$--orbit.
\end{enumerate}
We may insist on the first point because $\Stab_G(z)$ acts cocompactly on the set of edges adjacent to $z$, and the second because $H$ acts cocompactly on $\Gamma_H$ (and all of the intersection points $x_e$ lie at the midpoint of some edge in $\widetilde{X}^{(1)}$).  We henceforth make the above three assumptions on $\mc{E}$ (with extra assumptions to come, also ensured by passing to a further infinite subset).

We claim that for all $e\in \mc{E}$, $\left| Q_e \cap \Stab_G(z)\right| <\infty$. Indeed, the $\Stab_G(z)$ orbit of $e$ is infinite, so $\Stab_G(e)$ has infinite index in $\Stab_G(z)$, and because the action is relatively geometric $\Stab_G(e)$ is finite. Since $Q_e \cap \Stab_G(z) \subseteq \Stab_G(e)$, we have the required claim.

Fix $e_1 \in \mc{E}$, with dual hyperplane $W_1 = W_{e_1}$, and let $Q_1 = Q_{e_1} = \Stab_G(W_1)$, and $x_1 = x_{e_1} \in W_1 \cap \Gamma_H$.

For each $e \in \mc{E}$, choose $p_e \in \Stab_G(z)$ so $p_e \cdot e = e_1$, and note that $p_e \cdot W_e = W_1$.  Since $Q_1$ acts cocompactly on $W_1$, and we have $p_e \cdot x_e \in W_1$, there is an infinite subset $\mc{E}' \subset \mc{E}$ so that for all $f_1,f_2 \in \mc{E}'$ there exists $q_{f_1,f_2} \in Q_1$ so that 
$q_{f_1,f_2} p_{f_1} \cdot x_{f_1} = p_{f_2} \cdot x_{f_2}$, which is to say that
$p_{f_2}^{-1} q_{f_1,f_2} p_{f_1} \cdot x_{f_1} = x_{f_2}$.  We have already ensured that $x_{f_1}$ and $x_{f_2}$ lie in the same $H$--orbit, so let $h_{f_1,f_2} \in H$ be so that $h_{f_1,f_2} \cdot x_{f_1} = x_{f_2}$.  We now see that $p_{f_2}^{-1} q_{f_1,f_2} p_{f_1} h_{f_1,f_2}^{-1} \in \Stab_G(x_{f_2})$.

Now, $x_{f_2}$ is the midpoint of an edge dual to the unique hyperplane $W_{f_2}$, so clearly $\Stab_G(x_{f_2}) \le Q_{f_2}$, so there exists $s_{f_1,f_2} \in Q_{f_2}$ so that
\[	p_{f_2}^{-1} q_{f_1,f_2} p_{f_1} h_{f_1,f_2}^{-1} = \left( s_{f_1,f_2} \right)^{-1}	.	\]
From this it follows immediately that 
\[	h_{f_1,f_2}^{-1} s_{f_1,f_2} p_{f_2}^{-1} q_{f_1,f_2} p_{f_1} = 1 .	\]

We now translate to the cusped space $\mc{C} = \mc{C}(G,\mc{P})$ for $(G,\mc{P})$ (see Definition~\ref{def:cusped space}).  Recall from Definition~\ref{d:RH} that the cusped space is $\nu$--hyperbolic for some $\nu$, and that it contains a copy of the Cayley graph of $G$.  So elements of $G$ are vertices in $\mc{C}$ (the other vertices lie at some positive depth and lie in combinatorial horoballs stabilized by the conjugates of elements of $\mc{P}$).

Fix $f_2\in \mc{E}'$, and let $p = p_{f_2}$.  Let $R_H,$ $R_{Q_1}$ and $R_{Q_{f_2}}$ be the constants obtained by applying Proposition~\ref{prop:MMP} to $H,\,Q_1,$ and $Q_{f_2}$, respectively. 

We now consider other edges $f \in \mc{E}'$ and make the following simplifying notational choices: $q_f = q_{f,f_2} \in Q_1$, $h_f = h_{f,f_2} \in H$, $s_f = s_{f,f_2} \in Q_{f_2}$.  Therefore, the above equation becomes
\[	h_f^{-1} s_f p^{-1} q_f p_f = 1	.	\]

We make some observations.  First, $s_f \in Q_{f_2}$ and $h_f = h_{f,f_2}$ was chosen so that $h_f \cdot  x_{f} = x_{f_2}$.  Therefore, $h_f \cdot W_f = W_{f_2}$ and  $h_f^{-1} s_f h_f \in Q_f$.  As above, $Q_f \cap \Stab_G(z)$ and $Q_1 \cap \Stab_G(z)$ are both finite. Since $p_f^{-1} \in \Stab_G(z)$, $Q_1^{p_f^{-1}} \cap \Stab_G(z)$ is also finite.  Finally note that $q_f \in Q_1$.

Now, consider a geodesic pentagon in $\mc{C}$ with vertices $1, h_f\inv, h_f\inv s_f, h_f\inv s_fp^{-1}, h_f\inv s_fp^{-1}q_f$, so the sides are labelled (in order) $h_f\inv$, $s_f$, $p^{-1}$, $q_f$, $p_f$.
Let $B_z$ be the horoball stabilized by $\Stab_G(z)$.  
We claim there exists $R_{\mc{C}}>0$ not depending on $f$ so that $p_f$ penetrates $B_z$ to a depth at most $R_{\mc{C}}$.

Let $\alpha$ be any point on the edge labeled $p_f$. By subdividing the pentagon into 3 triangles and applying a standard hyperbolic geometry argument, $\alpha$ is at most $3\nu$ from one of the other four sides of the geodesic pentagon. Therefore, to prove the claim, it suffices to show that each of the other four sides penetrate $\Stab_G(z)$ to some bounded depth not depending on $f$.  
The length of the side labeled $p$ does not depend on $f$.
Therefore, the claim reduces to proving that the sides labeled $h_f\inv$, $q_f$ and $s_f$ penetrate $B_z$ to a depth bounded independent of $f$. 
%
%

We are in the case where $A = H\cap \Stab_G(z)$ is finite, and the side of the pentagon labeled $h_f\inv$ has endpoints in $H$. Therefore, this side penetrates $B_z$ to a depth at most $R_H$.

Above, we observed $Q_1 \cap \Stab_G(z)$ is finite, so a geodesic with endpoints $q_f^{-1}$ and $1$ penetrates $B_z$ to a depth at most $R_{Q_1}$. Translating this geodesic by $p_f^{-1} = h_f^{-1}s_fp^{-1}q_f \in \Stab_G(z)$ on the left implies that the side of the pentagon labeled $q_f$ penetrates $B_z$ to a depth at most $R_{Q_1}$.

Recall that $|Q_f \cap \Stab_G(z)|<\infty$. Therefore, we have:
\[|Q_{f_2} \cap \Stab_G(z)^{h_f}| = |Q_{f_2}^{h_f^{-1}} \cap \Stab_G(z)| = |Q_f \cap \Stab_G(z)|<\infty.\]
Thus, any geodesic with endpoints $1$ and $s_f\in Q_{f_2}$ penetrates the horoball stabilized by $\Stab_G(z)^{h_f}$ to a depth of at most $R_{Q_{f_2}}$.
Translating by $h_f\inv$,  
the side of the pentagon labeled $s_f$ penetrates $\Stab_G(z)$ to a depth at most $R_{Q_{f_2}}$.  
Hence, we have proved the claim that the side of the pentagon labeled $p_f$ penetrates $B_z$ to a bounded depth not depending on the choice of $f$. 

The bounded depth of $p_f$ in $B_z$ implies a bound on the distance (independent of $f$) 
between $1$ and $p_f$ in the naturally embedded copy of the Cayley graph of $G$ in $C$. 
Since $G$ is finitely generated, there are only finitely many possibilities for $p_f$. 
Recall $p_f \cdot W_f = W_1$, so we obtain a contradiction that implies $\mc{E}$ is finite, as required.

To finish the proof of Theorem~\ref{t:CCC}, we finally prove that the $H$--action on $Z = \Hull(\Gamma_H)$ is cocompact.

For all $y \in Z^{(0)}$ we have $d_{\widetilde{X}}(y,\Gamma_H) \le D$.  
Let $\nu_y$ be a shortest path in $\widetilde{X}^{(1)}$ from $\Gamma_H$ to $y$ so that $|\nu_y|\le D$.  
Also, $\leftQ{\Gamma_H}{H}$ is compact, so up to the $H$--action there are only finitely many possibilities for the starting point of $\nu_y$.  Also, $\nu_y \subseteq Z$, so by the claim each vertex of $\nu_y$ has only finitely many $H$--orbits of edges in $Z$ adjacent to it.  It now follows that the number of $H$--orbits of paths $\nu_y$ is finite, meaning that $\leftQ{Z}{H}$ is compact, as required.
\end{proof}

To ensure that these convex cores are compatible with the relatively hyperbolic geometry, we want to ensure they satisfy a condition similar to fullness.

\begin{cor}\label{L: peripherally full natural}
Let $(G,\cal{P})$ be a group pair acting relatively geometrically on a \CAT$(0)$ cube complex $\widetilde{X}$, let $H$ be a full relatively quasi-convex subgroup. For any compact $K_0 \subseteq \widetilde{X}$ there exists a convex $H$--invariant subset $\widetilde{Y}_{H,K_0} \subset \widetilde{X}$ so that (i) $K_0 \subseteq \widetilde{Y}_{H,K_0}$; (ii) For each $P \in \mc{P}$ we have $C(P) \subset \widetilde{Y}_{H,K_0}$; and (iii) $\leftQ{\widetilde{Y}_{H,K_0}}{H}$ is compact.  Moreover, the $H$--action on $\widetilde{Y}_{H,K_0}$ is relatively geometric.
\end{cor}

\begin{proof}
There are finitely many $P \in \mc{P}$ and by Proposition~\ref{p:C compact} each $C(P)$ is compact, so we can choose $K = K_0 \cup \bigcup_{P\in\mc{P}} C(P) $ and apply Theorem~\ref{t:CCC}.  That the $H$--action is relatively geometric follows immediately from the following facts: (i) the $G$--action on $\widetilde{X}$ is relatively geometric; (ii) $H$ is full relatively quasi-convex; and (iii) $\leftQ{\widetilde{Y}_{H,K_0}}{H}$ is compact.
\end{proof}

\section{Dehn Filling and cube complexes} \label{s:Dehn}
In previous papers such as \cite{AgolVirtualHaken,WiseManuscript,GMOmnibus,RelCannon} the combination of (relatively) hyperbolic groups acting on \CAT$(0)$ cube complexes and the behavior under Dehn filling has yielded very powerful tools.  We continue in this theme, in the context of relatively geometric actions.  
 
The following result is \cite[Proposition 2.3]{RelCannon}, and is an immediate consequence of \cite[Corollary 6.6]{GMOmnibus}. 
 It follows immediately from the definition of relatively geometric that there exists a family $\mc{Q}$ as in the statement below.
\begin{proposition} \cite[Proposition 2.3]{RelCannon} \label{prop:fill CAT(0)}
Suppose $(G,\cal{P})$ is relatively hyperbolic and that $G$ admits a relatively geometric action on a \CAT$(0)$ cube complex $\widetilde{X}$.  Let $\cal{Q}$ be a collection of finite-index subgroups of elements of $\cal{P}$ so that any infinite cell stabilizer contains a conjugate of an element of $\cal{Q}$.  For sufficiently long $\cal{Q}$--fillings 
\[	G \to \overline{G} = G/K	\]
of $(G,\cal{P})$, the quotient $\leftQ{\widetilde{X}}{K}$ is a \CAT$(0)$ cube complex.
\end{proposition}

\begin{cor} \label{cor:fill to get VS hyp}
 Let $(G,\cal{P})$ and $\cal{Q}$ be as in Proposition~\ref{prop:fill CAT(0)}.  For sufficiently long $\cal{Q}$--fillings $G \to \overline{G} = G\left( N_P \mid P \in \cal{P} \right)$, where each $P/N_P$ is virtually special and hyperbolic, the group $\overline{G}$ is virtually special and hyperbolic.
 
 In particular, for sufficiently long peripherally finite $\cal{Q}$--fillings, $\overline{G}$ is hyperbolic and virtually special.\footnote{Recall a group is \emph{virtually special} if it has a finite-index subgroup which admits a proper and cocompact action on a \CAT$(0)$ cube complex with quotient a special cube complex.}
\end{cor}
\begin{proof}
By Theorem~\ref{t:RHDF}, for sufficiently long fillings $G \to \overline{G} = G\left( N_P \mid P \in \cal{P} \right)$ the natural map $P/N_P \to \overline{G}$ is an embedding for each $P \in \mc{P}$ and the pair $\left( \overline{G}, \left\{ P/N_P \mid P \in \mc{P} \right\} \right)$ is relatively hyperbolic.  Thus, if each $P/N_P$ is hyperbolic then $\overline{G}$ itself is a hyperbolic group.  Since $(\overline{G}, \{ P/N_P \})$ is relatively hyperbolic, the subgroups $P/N_P$ of $\overline{G}$ are quasi-convex.
 
By Proposition~\ref{prop:fill CAT(0)} for sufficiently long $\mc{Q}$--fillings the space $\overline{X} := \leftQ{\widetilde{X}}{K}$ is a CAT$(0)$ cube complex, and $\overline{G} = G/K$ acts on $\overline{X}$.
The quotient $\leftQ{\overline{X}}{\overline{G}}$ has the same underlying space as $\leftQ{\widetilde{X}}{G}$, so it is compact.

The cell stabilizers for the $\overline{G}$--action on $\overline{X}$ are finite-index subgroups of the parabolic subgroups, and so these cell stabilizers are also quasi-convex in $\overline{G}$.  Therefore, if each $P/N_P$ is hyperbolic and virtually special then $\overline{G}$ is virtually special by \cite[Theorem D]{GMOmnibus}.
\end{proof}

We now prove Theorem~\ref{t:RGA} from the introduction.  Recall the statement.

\RGAGol*
\begin{proof}
Because the elements of $\mc{P}$ are residually finite, by Proposition~\ref{prop:fill CAT(0)} and Corollary~\ref{cor:fill to get VS hyp} there is a peripherally finite filling $G \to \overline{G} = G/K$ so that $\overline{G}$ is virtually special and $\overline{X} = \leftQ{\widetilde{X}}{K}$ is a CAT$(0)$ cube complex.
Thus, there is a finite-index subgroup $\overline{G}_0 \le \overline{G}$ so that $\leftQ{\overline{X}}{\overline{G}_0}$ is a special cube complex.  Let $G_0$ be the (finite-index) pre-image of $\overline{G}_0$ in $G$, and observe that the underlying space of $\leftQ{\widetilde{X}}{G_0}$ is the same as the underlying space of $\leftQ{\overline{X}}{\overline{G}_0}$.
\end{proof}

\subsection{Cores map to cores under suitable fillings}

Suppose $(G,\cal{P})$ acts relatively geometrically on the \CAT$(0)$ cube complex $\widetilde{X}$ with residually finite peripherals. 
Let $\mathcal{Q}$ be the set of subgroups from Proposition~\ref{prop:fill CAT(0)}, so that for sufficiently long $\mc{Q}$--fillings $G \to G/K$ the space $\leftQ{\widetilde{X}}{K}$ is a CAT$(0)$ cube complex.
Let $H$ be a full relatively quasi-convex subgroup of $G$ and let $\widetilde{Y} \subset \widetilde{X}$ be a convex $H$--invariant and $H$--cocompact sub-complex, the existence of which is guaranteed by Theorem~\ref{t:CCC}.

The following summarizes a collection of known and straightforward results about the existence of certain well-controlled Dehn fillings.
\begin{proposition}\label{prop:lots of filling properties}
For sufficiently long $\left( \mc{Q} \cup \{ H \} \right)$--fillings $G \to G / K$, the following statements hold:
\begin{enumerate}
\item\label{item:RHDF} The induced maps from each $P/N_P$ to $G/K$ are injective, and if $\overline{P}$ is the image of $P/N_P$ then $\left( G/K , \left\{ \overline{P} \mid P \in \mc{P} \right\} \right)$ is relatively hyperbolic;
\item\label{item:induced} If $K_H \unlhd H$ is the kernel of the induced filling of $H$ then $K_H = K \cap H$;
\item\label{item:RQC} If $\overline{H}$ denotes the image of $H$ in $G/K$, and $\overline{\mc{D}}$ is the collection of images of elements of $\mc{D}$ then $\left(\overline{H},\overline{\mc{D}} \right)$ is relatively quasi-convex in $\left( G/K, \left\{ \overline{P} \right\} \right)$;
\item\label{item:image cube} $\overline{X} = \leftQ{\widetilde{X}}{K}$ is a \CAT$(0)$ cube complex;
\item\label{item:H cube} $\overline{Y} = \leftQ{\widetilde{Y}}{K_H}$ is a \CAT$(0)$ cube complex;
\item\label{item:inj_on_int_para} If $Q_1$ and $Q_2$ are distinct maximal parabolic subgroups then $K \cap Q_1 \cap Q_2 = \{ 1 \}$;
\item\label{item:no finite stabilizers} If $S$ is a finite cell stabilizer, then $S\cap K = \emptyset$. 
\item\label{item:immersion} The induced map $\overline{Y} \to \overline{X}$ is an immersion. 
\end{enumerate}
\end{proposition}
\begin{proof}
Each of the items can be shown to hold for sufficiently long $\left( \mc{Q} \cup \{ H \} \right)$--fillings, and then we can take a single filling to satisfy them all.  Thus, we explain how to ensure each of the items individually.
Item~\ref{item:RHDF} follows from Theorem~\ref{t:RHDF}.  Items~\ref{item:induced} and \ref{item:RQC} follow from Theorem~\ref{t:Dehn QC}.  Item~\ref{item:image cube} follows from Proposition~\ref{prop:fill CAT(0)}.  For Item~\ref{item:H cube}, note that the $H$--action on $\widetilde{Y}$ is relatively geometric by Corollary~\ref{L: peripherally full natural}, so this item also follows from Proposition~\ref{prop:fill CAT(0)}.

Consider Item~\ref{item:inj_on_int_para}.  By \cite[Lemma~2.6]{GM-QCDF} there exists $M>0$ so that if $Q_1,Q_2$ are maximal parabolics,  then $Q_1\cap Q_2$ acts freely on a subset of the Cayley graph of $G$ of diameter at most $M$. Up to the action of $G$ there are only finitely many such subgraphs of the Cayley graph, so there are only finitely many $G$--conjugacy classes of intersections of maximal parabolic subgroups $Q_1 \cap Q_2$. Recall that $|Q_1\cap Q_2|<\infty$.
  Therefore, ensuring Item~\ref{item:inj_on_int_para} involves excluding finitely many elements from the filling kernels, so this item holds for sufficiently long fillings. For Item~\ref{item:no finite stabilizers}, there are finitely many conjugacy classes of finite stabilizers by cocompactness of the action. Item~\ref{item:no finite stabilizers} then follows by again excluding finitely many elements from the filling kernels.

Finally, we prove that we can ensure Item~\ref{item:immersion}.  Consider the set of cells $\tau \in \widetilde{Y}$ so that $\Stab_G(\tau)$ is infinite, but $\Stab_H(\tau)$ is finite.  Since $\leftQ{\widetilde{Y}}{H}$ is compact, there are only finitely many cells $\rho_1, \ldots, \rho_k$ in $\widetilde{Y}$ which contain $\tau$.  For each such distinct pair $\rho_i, \rho_j$ so that $\Stab_G(\rho_i)$ is finite, let $\mc{F}_{i,j}$ be the (possibly empty) finite set of elements $g \in \Stab_G(\tau)$ so that $g \cdot \rho_i = \rho_j$.  There are only finitely many $H$--orbits of cells in $\widetilde{Y}$, so by Theorem~\ref{t:RHDF} for sufficiently long fillings  $G \to G/K$ we have $\mc{F}_{i,j} \cap K = \emptyset$ for all such $i,j$ (and $\tau$).  Now suppose that $G \to G/K$ is such a $\left( \mc{Q} \cup \{ H \} \right)$--filling so that also Items~\ref{item:induced}, \ref{item:image cube} and \ref{item:H cube} hold.  Further, by taking a longer filling if necessary, we suppose that for any cell $\kappa$ in $X$ so that $\Stab_G(\kappa)$ is finite we have $K \cap \Stab_G(\kappa) = \emptyset$. In order to obtain a contradiction, suppose that $\overline{Y} \to \overline{X}$ is not an immersion and let $\overline{\sigma}_1$ and $\overline{\sigma}_2$ be adjacent (distinct) cells in $\overline{Y}$ with the same image in $\overline{X}$.    Note that $\overline{\sigma}_1 \cap \overline{\sigma}_2$ is a cell in $\overline{Y}$.  We may lift to cells $\sigma_1, \sigma_2$ in $\tilde{Y}$ with $\tau = \sigma_1 \cap \sigma_2$ a cell in $\tilde{Y}$.  Since the images of $\overline{\sigma}_1$ and $\overline{\sigma}_2$ are equal in $\overline{X}$ there exists $k \in K$ so that $k \cdot \sigma_1 = \sigma_2$.  Since $\overline{X}$ is a CAT$(0)$ cube complex, we have $k \cdot \tau = \tau$, so $k \in \Stab_G(\tau)$.  If $\Stab_G(\tau)$ is finite, then we have $K \cap \Stab_G(\tau) = \emptyset$, so there is no such $k$.  Therefore, we may assume that $\Stab_G(\tau)$ is infinite.  If $\Stab_H(\tau)$ is also infinite, then because we have a relatively geometric action, and the filling is an $H$--filling, we have $K \cap \Stab_G(\tau) \le H$, which means that $\overline{\sigma}_1 = \overline{\sigma}_2$ in $\overline{Y}$, contrary to our choice.  

We are left with the possibility that $\Stab_G(\tau)$ is infinite, but $\Stab_H(\tau)$ is finite. If $\Stab_G (\sigma_i)$ is infinite, it is commensurable into $\Stab_G(\tau)$. Then, since the filling is a $\mc{Q}$--filling, $K\cap \Stab_G(\tau) \le \Stab_G(\sigma_i)$,  which contradicts the equation $k \cdot \sigma_1 = \sigma_2$.  
Finally, suppose that $\Stab_G(\sigma_1)$ and $\Stab_G(\sigma_2)$ are both finite, but that $\Stab_G(\tau)$ is infinite (and $\Stab_H(\tau)$ is still finite).  In this case, we have $\sigma_1 = \rho_i$ and $\sigma_2 = \rho_j$ for some $i$ and $j$, where the $\rho_i$ and $\rho_j$ are as chosen above. 
If $\Stab_G(\rho_i)$ is infinite and we clearly have $k \in \mc{F}_{i,j}$, contradicting the assumption that $\mc{F}_{i,j}\cap K =\emptyset$ about our filling.  This completes the proof.
\end{proof}

We now prove the main result of this section, which immediately implies Theorem~\ref{t:fill hull} from the introduction.
 \begin{theorem}\label{completion loc iso}
For all sufficiently long $\left(\mc{Q} \cup \{ H \} \right)$--fillings $G\to \overline{G}$, the immersion $f:\overline{Y}\to \overline{X}$ is an embedding, and (the image of) $\overline{Y}$ is convex in $\overline{X}$.
\end{theorem}
\begin{proof}
Let $\mc{Q}$ be the set of subgroups as in Proposition~\ref{prop:fill CAT(0)}.  Let $\mc{D}$ be the induced peripheral structure for $H$.  Let $\mc{Q}_H = \mc{Q} \cup \{ H \}$.

We only consider $\mc{Q}_H$--fillings which are long enough to satisfy the conclusions of Proposition~\ref{prop:lots of filling properties}.  We impose a further condition below so that if both conditions are simultaneously satisfied then $\overline{Y}$ is convex in $\overline{X}$.

For a cell $\sigma \subseteq \widetilde{X}$, denote the $G$--orbit of $\sigma$ by $\orbit{\sigma}$.  Moreover, for a cell $\tau \subseteq \widetilde{Y}$, let $\orbit{\tau}_H$ denote the $H$--orbit of $\tau$.
Consider the set $\mc{S}$ of all pairs $\left( \orbit{e_1}_H, \orbit{e_2}_H \right)$, where $e_1$ and $e_2$ are (oriented) edges in $\widetilde{Y}$ so that there exist $e_1' \in  \orbit{e_1}_H$ and $e_2' \in  \orbit{e_2}_H$ so that $e_1'$ and $e_2'$ have the same initial vertex, and so that $e_1'$ and $e_2'$ bound the corner of a square in $\widetilde{X}$.  Since $\leftQ{\widetilde{Y}}{H}$ is compact, $\mc{S}$ is finite.  By rechoosing $e_2$ if necessary, we always assume that if $\left( \orbit{e_1}_H, \orbit{e_2}_H \right) \in \mc{S}$ then the edges $e_1, e_2$ in $\widetilde{Y}$ share the same initial vertex.

Let $\mc{S}'$ denote the set of all $\rho = \left( \orbit{e_1}_H, \orbit{e_2}_H \right) \in \mc{S}$ so that
\begin{enumerate}
\item $e_1$ and $e_2$ do not bound a square in $\widetilde{Y}$; and
\item $\Stab_G(e_1)$ and $\Stab_G(e_2)$ are finite.
\end{enumerate}
Let $\rho = \left( \orbit{e_1}_H, \orbit{e_2}_H \right) \in \mc{S}'$.
Since $\widetilde{Y}$ is convex in $\widetilde{X}$, $e_1$ and $e_2$ do not bound a square in $\widetilde{X}$ either.  Furthermore, since $\Stab_G(e_1)$ is finite and because the action of $G$ on $\widetilde{X}$ is cocompact, there are only finitely many squares $f_1, \ldots, f_k$ adjacent to $e_1$ in $\widetilde{X}$.  For each $f_i$, let $\hat{e}_2^i$ be the edge of $f_i$ which shares the initial vertex of $e_1$. 

Since $\Stab_G(e_2)$ is finite there are only finitely many elements $g \in G$ so that for some $1 \le i \le n$ we have $g \cdot \hat{e}_2^i = e_2$.  Let $\mc{F}_\rho$ denote the set of all such $g$, and let $\mc{F}$ be the union of the $\mc{F}_\rho$ over all $\rho \in \mc{S'}$. Note that each $\mc{F}_\rho$ is finite and $\mc{S'}\subseteq\mc{S}$ is finite, so $\mc{F}$ is finite.

By Theorem~\ref{t:RHDF} for sufficiently long fillings $G \to \rightQ{G}{K}$, $K \cap \mc{F} = \emptyset$.
Fix now a $\mc{Q}_H$--filling $G \to \rightQ{G}{K}$ long enough to satisfy the conclusion of Proposition~\ref{prop:lots of filling properties}, and also so that $K \cap \mc{F} = \emptyset$.

We claim that with such a filling, and the notation as above, the subspace $\overline{Y}$ is convex in $\overline{X}$.
In order to obtain a contradiction, suppose that there are edges $\overline{e}_1, \overline{e}_2 \in \overline{Y}$ so that $\overline{e}_1$ and $\overline{e}_2$ do not bound a square in $\overline{Y}$ but they do bound a square $\overline{f}$ in $\overline{X}$.

Lift $\overline{e}_1$ to an edge $e_1$ in $\widetilde{Y}$, and $\overline{e}_2$ to an edge $e_2$ with the same initial vertex, $x$ say, as $e_1$.  Let $f$ be a lift of $\overline{f}$ to $\widetilde{X}$ so that $e_1$ is an edge on the boundary of $f$.  Since $\overline{f} \not\in \overline{Y}$, we see that $f \not\in \widetilde{Y}$.   Moreover, there is no square in $\widetilde{X}$ with $e_1$ and $e_2$ at a corner, because $\widetilde{Y}$ is convex in $\widetilde{X}$.

Let $\hat{e}_2$ be the edge on the boundary of $f$ with initial point $x$.  Because the images of $\hat{e}_2$ and $e_2$ are both $\overline{e}_2$ in $\overline{X}$, there exists $k \in K \cap \Stab_G(x)$ so $k \cdot \hat{e}_2 = e_2$.

First suppose that $\Stab_G(e_1)$ is infinite.  Since the $G$--action on $\widetilde{X}$ is relatively geometric, $\Stab_G(e_1)$ is finite-index in $\Stab_G(x)$, and since $G \to G/K$ is a $\mc{Q}_H$--filling, $\Stab_G(x) \cap K \le \Stab_G(e_1)$.  Therefore, $k \cdot e_1 = e_1$, so $k\cdot f$ is a $2$--cell which has $e_1$ and $e_2$ as a corner.  This is a contradiction, so $\Stab_G(e_1)$ is finite.

Now suppose that $\Stab_G(e_2)$ is infinite.  In this case, $K \cap \Stab_G(x) \le \Stab_G(e_2)$, which contradicts the equation $k \cdot \hat{e}_2 = e_2$.  Therefore, $\Stab_G(e_2)$ is also finite.

Using $\orbit{\cdot}$ to denote orbits as above, $\rho = \left( \orbit{e_1}_H, \orbit{e_2}_H \right) \in \mc{S}'$.  Then $k \in \mc{F}_\rho \subseteq \mc{F}$, contradicting $K \cap \mc{F} = \emptyset$.   This contradiction proves the immersion $\overline{Y} \to \overline{X}$ is locally convex.  Since $\overline{X}$ is a \CAT$(0)$ cube complex, a locally convex immersion is an embedding with convex image, completing the proof.
 \end{proof}

\section{Completion and retraction with complexes of groups}\label{s:CCR}

In this section we prove Theorem~\ref{t:VVR}.  Our approach is to prove a relatively geometric analogue of the canonical completion and retraction due to Haglund and Wise \cite{HW08} (see Theorem~\ref{t:Completion_retraction} below).  We prove this by applying a Dehn filling as in Proposition~\ref{prop:lots of filling properties} and Theorem~\ref{completion loc iso}, applying Agol's Theorem \cite[Theorem 1.1]{AgolVirtualHaken}, passing to a carefully chosen finite-index subgroup, applying the Haglund-Wise construction, and then noting that the retraction of cube complexes induces a retraction of complexes of groups.  For the basic  theory of complexes of groups, we refer to \cite[III.$\mc{C}$]{BridsonHaefliger}.  

\begin{remark}
The completion and retraction we construct in Theorem~\ref{t:Completion_retraction} below relies on a particular Dehn filling, and so is not as ``canonical" as that of Haglund--Wise.
\end{remark}

\subsection{The canonical completion for relatively geometric complexes of groups}
The following result summarizes Haglund and Wise's construction of the canonical completion and retraction for special cube complexes.
\begin{theorem} \cite[$\S6$]{HW08}, \cite[$\S3$]{HW2012} \label{t:CCR}
Suppose that $A$ and $B$ are cube complexes, that $B$ is special, that $A$ is compact and that $f \co A \to B$ is a locally convex combinatorial map. There exists a pair of cube complexes $\boldsymbol{C}(A,B)$ and $\boldsymbol{C}_{\boxminus}(A,B)$, along with
\begin{enumerate}
\item A homeomorphism $s \co  \boldsymbol{C}_{\boxminus}(A,B) \to \boldsymbol{C}(A,B)$;
\item A finite (combinatorial) covering $p\co \boldsymbol{C}(A,B) \to B$;
\item A (combinatorial) embedding $i \co A \to \boldsymbol{C}_{\boxminus}(A,B)$ so that $s \circ i$ is a (combinatorial) embedding, and $p \circ s \circ i = f$;
\item A cellular retraction $r \co \boldsymbol{C}_{\boxminus}(A,B) \to A$ (so $r \circ i = \operatorname{Id}_{A}$).
\end{enumerate}
The following diagram commutes:

\medskip
\
\centerline{
\xymatrix{
& \boldsymbol{C}_{\boxminus}(A,B) \ar@{>}@/_/[dl]_{r} \ar@{>}[r]^{s} & \boldsymbol{C}(A,B)\ar@{>}[d]^p \\
A \ar@{>}[ur]_{i} \ar@{>}[rr]_{f}  & & B
}}
\medskip
\end{theorem}
As described in \cite[Definition 3.5]{HW2012}, $\boldsymbol{C}_{\boxminus}(A,B)$ is obtained from $\boldsymbol{C}(A,B)$ by sub-dividing certain cubes (some of those outside of the image of $A$), and the map $r$ maps each cube onto a face of a cube in the target by orthogonal projection.

Our goal is to set up a situation of complexes of groups so that the underlying spaces are arranged in a diagram as above.  We then explain how to turn the corresponding maps into morphisms of complexes of groups, giving Theorem~\ref{t:VVR}.  We now record the set up to our construction in the following assumption, which builds on Assumption~\ref{as:RG}.

\begin{assumption} \label{as:CCR}
Suppose that $\left( G,\mc{P} \right)$ is relatively hyperbolic, that $\widetilde{X}$ is a \CAT$(0)$ cube complex which admits a relatively geometric action of $G$ with respect to $\mc{P}$, and let $X = \leftQ{\widetilde{X}}{G}$.  Further, suppose all elements of $\mc{P}$ are residually finite.

Let $H$ be a full relatively quasi-convex subgroup, and let $\left( H , \mc{D} \right)$ be the peripheral structure on $H$ induced from $(G,\mc{P})$.   If $D \in \mc{D}$ and $D \le P_D^{c_D}$ for some $c_D \in G$ and $P_D \in \mc{P}$, let $C(P_D)$ be the sub-complex of $\widetilde{X}$ associated to $P_D$ from Definition~\ref{def:para stab}.  Let $\widetilde{Y} \subset \widetilde{X}$ be a convex $H$--invariant and $H$--cocompact subcomplex as in the conclusion of Theorem~\ref{t:CCC} so that for each $D \in \mc{D}$ we have $c_D \cdot C(P_D) \subset \widetilde{Y}$ (this can be ensured by Proposition~\ref{p:C compact} and Theorem~\ref{t:CCC}).

Let $\mc{Q}$ be a collection of subgroups as in the hypotheses of Proposition~\ref{prop:fill CAT(0)}.

Let $\pi \co G \to \overline{G} = G\left( N_P \mid P \in \mc{P} \right) = G/K$ be a peripherally finite $\left( \mc{Q} \cup \{ H \} \right)$--filling which satisfies the conclusions of Proposition~\ref{prop:lots of filling properties} and Theorem~\ref{completion loc iso}.

Let $\overline{H}$ be the image of $H$ in $\overline{G}$, and let $\overline{X} = \leftQ{\widetilde{X}}{K}$, $\overline{Y} = \leftQ{\widetilde{Y}}{K_H}$ be as in Proposition~\ref{prop:lots of filling properties}.\end{assumption}
For the remainder of this section, we make Assumption~\ref{as:CCR}.

%
%
%
%

\begin{proposition}
The group $\overline{G}$ is hyperbolic and virtually special.  In particular, it is residually finite and virtually torsion-free.  Moreover, there is a finite-index torsion-free subgroup $\overline{G}_0 \le \overline{G}$ so that $\leftQ{\overline{X}}{\overline{G}_0}$ is a special cube complex.
\end{proposition}
\begin{proof}
Since the filling $G \to \overline{G}$ is peripherally finite, $\overline{G}$ is hyperbolic relative to finite groups, and hence is hyperbolic.  Moreover, $\overline{G}$ acts cocompactly on the \CAT$(0)$ cube complex $\overline{X}$ (since $\leftQ{\overline{X}}{\overline{G}}$ and $\leftQ{\widetilde{X}}{G}$ have the same underlying space).  Because the $G$--action on $\widetilde{X}$ is relatively geometric, it follows that stabilizers for the $\overline{G}$--action on $\overline{X}$ are finite.  Thus, the hyperbolic group $\overline{G}$ acts properly and cocompactly on the \CAT$(0)$ cube complex $\overline{X}$.  By Agol's Theorem, $\overline{G}$ is virtually special, and hence residually finite.
It is well known that any residually finite hyperbolic group is virtually torsion-free, so there is a torsion-free finite-index subgroup $\overline{G}_0 \le \overline{G}$ so that $\leftQ{\overline{X}}{\overline{G}_0}$ is a special cube complex, as required.
\end{proof}

Let $B = \leftQ{\overline{X}}{\overline{G}_0}$.  Since $\overline{G}_0$ is torsion-free, $\overline{G}_0 = \pi_1(B)$.
Define $\overline{H}_0 = \overline{H} \cap \overline{G}_0$, and note that $\overline{H}_0$ is torsion-free, so acts freely on $\overline{Y}$.  Let $A = \leftQ{\overline{Y}}{\overline{H}_0}$.  The convex embedding $\overline{Y} \to \overline{X}$ from the conclusion of Theorem~\ref{completion loc iso} descends to a locally convex immersion $f \co A \to B$.  Since $B$ is special, Theorem~\ref{t:CCR} applies to the map $f \co A \to B$, and we obtain the canonical completion $\boldsymbol{C}(A,B)$, and its subdivided version $\boldsymbol{C}_{\boxminus}(A,B)$ as in Theorem~\ref{t:CCR}.
Let $\overline{G}_1$ be the finite-index subgroup of $\overline{G}_0$ corresponding to the finite cover $\boldsymbol{C}(A,B) \to B$. 
By the construction of the canonical completion $\overline{H}_0 \le \overline{G}_1$.

Let $C = \boldsymbol{C}_{\boxminus}(A,B)$, so $\overline{G}_1 = \pi_1(C)$ (recall $\boldsymbol{C}_{\boxminus}(A,B)$ is homeomorphic to $\boldsymbol{C}(A,B)$). Let $\widetilde{C}$ be the \CAT$(0)$ universal cover of $C$.  Since $A \to C$  is an inclusion as a sub-complex, $\widetilde{Y}$ is a convex sub-complex of $\widetilde{C}$. 

Let $\widetilde{C}$ be the induced sub-divided version of $\widetilde{X}$, and let $G_1 = \pi^{-1}\left(\overline{G}_1\right)$, a finite-index subgroup of $G_0$, and $H_0 = H \cap G_1$, a finite-index subgroup of $H$ so $\pi(H_0) = \overline{H}_0$. The universal cover $\overline{C}$ of $C$ is a \CAT$(0)$ cube complex which is a sub-divided version of $\overline{X}$.

There is a $G_1$--action on $\widetilde{C}$ and an $H_0$--action on $\widetilde{Y}$, so that the underlying space of $\leftQ{\widetilde{C}}{G_1}$ is $C$ and the underlying space of $\leftQ{\widetilde{Y}}{H_0}$ is $A$.  Observe the $G_1$--action on $\widetilde{C}$ is relatively geometric, with respect to the induced peripheral structure $(G_1,\mc{P}_1)$ on $G_1$.

\begin{proposition}\label{prop:magic stabilizers}
The following properties for the $G_1$--action on $\widetilde{C}$ and the $H_0$--action on $\widetilde{Y}$ hold:
\begin{enumerate}
\item\label{eq:triv or max para} Stabilizers in $G_1$ of cells in $\widetilde{C}$ are either trivial or else maximal parabolic subgroups of $G_1$.
\item\label{eq:C(P) embeds} If $P_1 \in \mc{P}_1$ then the sub-complex $C(P_1)$ from Definition~\ref{def:para stab} has cells whose stabilizers are exactly $P_1$, and $C(P_1)$ embeds in $C$ under the quotient map.
\item\label{eq:equal stabs} If $\sigma \in \widetilde{Y}$ is a cell with nontrivial $H_0$--stabilizer then $\Stab_{H_0}(\sigma) = \Stab_{G_1}(\sigma)$.
\end{enumerate}
\end{proposition}

\begin{proof}
Stabilizers are already finite or finite-index in a maximal parabolic because the action is relatively geometric. 
Since $\bar{G}_1$ is torsion free and $\pi \co G \to \overline{G}$ is injective on finite stabilizers by Proposition~\ref{prop:lots of filling properties}.\eqref{item:no finite stabilizers}, $G_1$ has no nontrivial finite cell stabilizers.
The filling $\pi:G \to \overline{G}$ is a $\mc{Q}$--filling and $\pi$ is injective on the finite groups $\left\{ P/N_P \right\}$, so because $\bar{G}_1$ is torsion free any infinite stabilizer must be maximal parabolic in $G_1$.  This proves \eqref{eq:triv or max para}.

Item~\eqref{eq:C(P) embeds} follows because all non-trivial stabilizers are maximal parabolic in $G_1$ and maximal parabolic subgroups are malnormal. 

For Item~\eqref{eq:equal stabs}, suppose $\sigma \in \widetilde{Y}$ has $\Stab_{H_0}(\sigma) \ne \{ 1 \}$.  By Proposition~\ref{prop:lots of filling properties}.\eqref{item:no finite stabilizers} the map $H \to \overline{H}$ is injective on finite cell stabilizers.  Moreover, $\overline{H}_0$ is torsion-free, and $\overline{H}_0$ is the induced filling of $H_0$.  Therefore, $\Stab_{H_0}(\sigma)$ is infinite, and hence full parabolic.  
Let $Q$ be the maximal parabolic subgroup of $G_1$ containing $\Stab_{H_0}(\sigma)$. Then $Q = \Stab_{G_1}(\sigma)$ by Item~\eqref{eq:triv or max para}. 
Since $\overline{G}_1$ is torsion-free and $\pi$ is an $H$-filling, $\Stab_{H_0}(\sigma) = Q = \Stab_{G_1}(\sigma)$, as required. 
\end{proof}

The actions of $G_1$ on $\widetilde{C}$ and $H_0$ on $\widetilde{Y}$ give rise to a pair of complexes of groups $G(\mc{C})$ and $H(\mc{A})$, with underlying scwols $\mc{C}$ arising from $C$ and $\mc{A}$ arising from $A$, respectively.  The map $i \co A \to C$ gives rise to a (non-degenerate) morphism of scwols $f_i \co \mc{A} \to \mc{C}$ as in \cite[III.$\mc{C}$.1.5]{BridsonHaefliger}.    A complex of groups comes with a collection of data, one part of which is ``twisting elements" (see \cite[Definition III.$\mc{C}$.2.1]{BridsonHaefliger}.  A complex of groups is \emph{simple} if all the twisting elements are trivial (see \cite[III.$\mc{C}$.2.1]{BridsonHaefliger}).  To build the complexes of groups $G(\mc{C})$ and $H(\mc{A})$, follow the construction from \cite[$\S$ III.$\mc{C}$.2.9]{BridsonHaefliger}.  This construction involves some choices (of the lifts of objects, and of the elements $h_a$).  However, we make the following observation.

\begin{lemma}
We may make choices in the constructions of $G(\mc{C})$ and $H(\mc{A})$ so that both are simple complexes of groups.
\end{lemma}
\begin{proof}
Suppose $\sigma$ is an object in $\mc{C}$ whose stabilizer is nontrivial, then we lift to (the scwol $\mc{X}$ associated to) $X$ and obtain a nontrivial cell stabilizer in $G_1$.  By Proposition~\ref{prop:magic stabilizers}.\eqref{eq:C(P) embeds} the set of objects in $\mc{C}$ whose stabilizers intersect $\Stab(\sigma)$ nontrivially can be simultaneously lifted to $\mc{X}$ when defining $G(\mc{C})$, so for these cells the twisting elements can be chosen to be trivial.  For other cells, the stabilizer is trivial, so twisting elements are trivial.  The proof for $H(\mc{A})$ is the same.
\end{proof}

By \cite[Corollary III.$\mc{C}$.2.18]{BridsonHaefliger}, the $H_0$--equivariant inclusion $\widetilde{Y} \to \widetilde{X}$ induces a morphism of complexes of groups $\phi \co H(\mc{A}) \to G(\mc{C})$ over $f_i$.  We can, and do, consider the morphism $f_i$ to be inclusion, so that objects and arrows of $\mc{A}$ are contained in $\mc{C}$.  By choosing lifts of $\mc{A}$ before lifts of $\mc{C}$ when defining the complex of groups structures, we can ensure that if $\sigma$ is an object of $\mc{A}$, and $H_\sigma$ is nontrivial, then $H_\sigma = G_\sigma$, and the map $\phi_\sigma \co H_\sigma \to G_\sigma$ is the identity.

Now, let $r \co C \to A$ be the canonical retraction from Theorem~\ref{t:CCR}, and let $f_r \co \mc{C} \to \mc{A}$ be the associated morphism of scwols (since cells in $C$ may be collapsed under $r$, the morphism $f_r$ is probably degenerate).  

We now define a morphism of complexes of groups $\theta \co G(\mc{C}) \to H(\mc{A})$ over $f_r$.
Let $\sigma$ be an object of $\mc{C}$, and consider the object $f_r(\sigma) \in \mc{A}$.  

\begin{lemma} \label{lem:same or disj}
Either $H_{f_r(\sigma)} = G_\sigma$ or else $H_{f_r(\sigma)} \cap G_\sigma = \{ 1 \}$.
\end{lemma}
\begin{proof}
The subgroup $H_{f_r(\sigma)}$ is a cell stabilizer for the action of $H_0$ on $\tilde{Y}$, so if $H_{f_r(\sigma)} \ne \{ 1 \}$ then by Proposition~\ref{prop:magic stabilizers}.\eqref{eq:equal stabs} $H_{f_r(\sigma)}$ is a cell stabilizer for the action of $G_1$ on $\tilde{C}$. 
Cell stabilizers for the action of $G_1$ on $\tilde{C}$ are either trivial or maximal parabolic. Therefore, either $G_\sigma = H_{f_r(\sigma)}$ or they are distinct cell stabilizers in $G_1$ and have trivial intersection by Proposition~\ref{prop:lots of filling properties}.\eqref{item:inj_on_int_para} and the construction of $G_1$.
\end{proof}

In case $H_{f_r(\sigma)} = G_\sigma$ define $\theta_\sigma$ to be the identity map, and in case $H_{f_r(\sigma)} \cap G_\sigma = \{ 1 \}$ define $\theta_\sigma$ to be the trivial map.
For each arrow $a \in \mc{C}$, define the element $\theta(a)$ to be the identity element of $H_{t(f_r(a))}$.

This data defines the structure of a morphism of complexes of groups $\theta$, as follows from the next result (where $\psi_{f_r(a)}$ and $\psi_a$ are the homomorphisms in the complexes of groups $H(\mc{A})$ and $G(\mc{C})$ respectively).

\begin{lemma}
For each arrow $a$ of $\mc{C}$ we have 
\[	\psi_{f_r(a)} \theta_{i(a)} = \theta_{t(a)} \psi_{a}	.	\]
\end{lemma}
\begin{proof}
Suppose first that $G_{t(a)} \ne H_{t(f_r(a))}$. Then $G_{t(a)}\cap H_{t(f_r(a))} = \{1\}$ by Lemma~\ref{lem:same or disj}. If $\theta_{i(a)}$ is non-trivial, then $G_{i(a)} = H_{i(f_r(a))}$. 
The local maps $\psi_a,\psi_{f_r(a)}$ are inclusions, so $ G_{t(a)}\cap H_{t(f_r(a))}$ is non-trivial, a contradiction. Therefore, $\theta_{i(a)}$ must be trivial, and the lemma holds in this case.

On the other hand, suppose that $G_{t(a)} = H_{t(f_r(a))}$. Then $\theta_{t(a)}$ is the identity map.
Non-trivial cell stabilizers in $G_1$ are maximal parabolic by Proposition~\ref{prop:magic stabilizers}.\eqref{eq:triv or max para} and intersections of maximal parabolics are trivial by Proposition~\ref{prop:lots of filling properties}.\eqref{item:inj_on_int_para}. Therefore, $G_{i(a)}$ is trivial or $G_{i(a)} = G_{t(a)}$.  
In the first case, $\theta_{i(a)}$ is the trivial map and the lemma follows, so suppose $G_{i(a)} = G_{t(a)}$, a maximal parabolic subgroup. 
 Note that $\theta_{i(a)}\psi_a(G_{i(a)}) = G_{i(a)}$.  Also, we have (in this case) $G_{i(a)} = G_{t(a)} = H_{t(f_r(a))}$.  Moreover, $H_{t(f_r(a))}$ is some maximal parabolic in $H_0$, which in turn is finite-index in a maximal parabolic in $H$.  Because we chose $\widetilde{Y}$ to contain all of the $c_D \cdot C(P_D)$ in Assumption~\ref{as:CCR}, and because $\widetilde{Y}$ is $H$--invariant, we see that $C(G_{i(a)}) \subseteq \widetilde{Y}$.
 Then $i(a)$ is a cell in the image of $\leftQ{\tilde{Y}}{H_0} \subseteq \leftQ{\tilde{C}}{G_1}$ that is fixed by $f_r$.
Therefore, $\theta_{i(a)}$ is the identity map, and the lemma follows.  
\end{proof}

It follows immediately from the construction that $\theta \circ \phi = \operatorname{Id}_{H(\mc{A})}$.

The induced map on $\pi_1$ from $\phi$ is the inclusion $\iota \co H_0 \to G_1$.  Moreover, if $\rho \co G_1 \to H_0$ is $\rho = \pi_1(\theta)$ then the fact that $\theta \circ \phi = \operatorname{Id}_{H(\mc{A})}$ implies $\rho \circ \iota = \operatorname{Id}_{H_0}$.  This proves Theorem~\ref{t:VVR} from the introduction. 

For future use, we summarize the above construction in the following result (in the statement below, we elide the difference between a quotient space and the induced complex of groups).

\begin{theorem} \label{t:Completion_retraction}
Make Assumption~\ref{as:RG}, and suppose further that elements of $\mc{P}$ are residually finite.  Let $H$ be a full relatively quasiconvex subgroup of $G$. There exist
\begin{itemize}
\item a cocompact convex core $\tilde{Y}\subseteq \tilde{X}$ for $H$;
\item finite index subgroups $H_0\le H$ and $G_1\le G$, so $H_0 \le G_1$;
\item and a subdivision $\tilde{C}$ of $\tilde{X}$ with an embedding $\tilde{Y}\hookrightarrow \tilde{C}$;
\end{itemize}
together with morphisms of complexes of groups:
\[\phi: \leftQ{\tilde{Y}}{H_0} \to \leftQ{\tilde{C}}{G_1} \qquad \text{ and } \qquad \theta: \leftQ{\tilde{C}}{G_1} \to \leftQ{\tilde{Y}}{H_0}\]
so that
\begin{itemize}
\item the underlying map of $\phi$ is an embedding of special cube complexes
\item for each cell $\sigma$ of $\leftQ{\tilde{Y}}{H_0}$, either the local group $H_\sigma$ is trivial or $H_\sigma$ and the local group $G_{\phi(\sigma)}$ are equal and the map $\phi_\sigma:H_\sigma\to G_{\phi(\sigma)}$ is the identity map. 
\item $\theta\circ \phi$ is the identity morphism on $\leftQ{\tilde{Y}}{H_0}$.
\end{itemize}
\end{theorem}


\begin{thebibliography}{10}

\bibitem{AgolVirtualHaken}
Ian Agol.
\newblock The {V}irtual {H}aken {C}onjecture.
\newblock {\em Doc. Math.}, 18:1045--1087, 2013.
\newblock With an appendix by Agol, Daniel Groves, and Jason F. Manning.

\bibitem{AGM09}
Ian Agol, Daniel Groves, and Jason~F. Manning.
\newblock Residual finiteness, {QCERF} and fillings of hyperbolic groups.
\newblock {\em Geom. Topol.}, 13(2):1043--1073, 2009.

\bibitem{BridsonHaefliger}
Martin~R. Bridson and Andr{\'e} Haefliger.
\newblock {\em Metric spaces of non-positive curvature}, volume 319 of {\em
  Grundlehren der Mathematischen Wissenschaften [Fundamental Principles of
  Mathematical Sciences]}.
\newblock Springer-Verlag, Berlin, 1999.

\bibitem{CharneyCrisp}
Ruth Charney and John Crisp.
\newblock Relative hyperbolicity and {A}rtin groups.
\newblock {\em Geom. Dedicata}, 129:1--13, 2007.

\bibitem{RelCannon}
Eduard Einstein and Daniel Groves.
\newblock Relative cubulations and groups with a 2-sphere boundary.
\newblock {\em Compos. Math.}, 156(4):862--867, 2020.

\bibitem{FarbRelHypGroups}
Benson Farb.
\newblock Relatively hyperbolic groups.
\newblock {\em Geom. Funct. Anal.}, 8(5):810--840, 1998.

\bibitem{GM08}
Daniel Groves and Jason~F. Manning.
\newblock Dehn filling in relatively hyperbolic groups.
\newblock {\em Israel J. Math.}, 168:317--429, 2008.

\bibitem{GMOmnibus}
Daniel Groves and Jason~F. Manning.
\newblock Hyperbolic groups acting improperly.
\newblock arxiv.org:1808.02325, 2018.

\bibitem{GM20}
Daniel Groves and Jason~F. Manning.
\newblock Specializing cubulated relatively hyperbolic groups.
\newblock arXiv:2008.13677, 2020.

\bibitem{GM-QCDF}
Daniel Groves and Jason~F. Manning.
\newblock Quasiconvexity and {D}ehn filling.
\newblock {\em Amer. J. Math.}, 143(1):95--124, 2021.

\bibitem{Haglund2008}
Fr\'ed\'eric Haglund.
\newblock Finite index subgroups of graph products.
\newblock {\em Geom. Dedicata}, 135:167--209, 2008.

\bibitem{HW08}
Fr{\'e}d{\'e}ric Haglund and Daniel~T. Wise.
\newblock Special cube complexes.
\newblock {\em Geometric and Functional Analysis}, 17(5):1551--1620, 2008.

\bibitem{HW2012}
Fr\'{e}d\'{e}ric Haglund and Daniel~T. Wise.
\newblock A combination theorem for special cube complexes.
\newblock {\em Ann. of Math. (2)}, 176(3):1427--1482, 2012.

\bibitem{Hruska2010}
G.~Christopher Hruska.
\newblock Relative hyperbolicity and relative quasiconvexity for countable
  groups.
\newblock {\em Algebr. Geom. Topol.}, 10(3):1807--1856, 2010.

\bibitem{HruskaWise14}
G.~Christopher Hruska and Daniel~T. Wise.
\newblock Finiteness properties of cubulated groups.
\newblock {\em Compos. Math.}, 150(3):453--506, 2014.

\bibitem{HW2015}
Tim Hsu and Daniel~T. Wise.
\newblock Cubulating malnormal amalgams.
\newblock {\em Invent. Math.}, 199(2):293--331, 2015.

\bibitem{MMP}
Jason~F. Manning and Eduardo Mart\'{\i}nez-Pedroza.
\newblock Separation of relatively quasiconvex subgroups.
\newblock {\em Pacific J. Math.}, 244(2):309--334, 2010.

\bibitem{OsinRelHyp}
Denis~V. Osin.
\newblock Relatively hyperbolic groups: intrinsic geometry, algebraic
  properties, and algorithmic problems.
\newblock {\em Mem. Amer. Math. Soc.}, 179(843):vi+100, 2006.

\bibitem{osin:peripheral}
Denis~V. Osin.
\newblock Peripheral fillings of relatively hyperbolic groups.
\newblock {\em Invent. Math.}, 167(2):295--326, 2007.

\bibitem{PW}
Piotr Przytycki and Daniel~T. Wise.
\newblock Mixed 3-manifolds are virtually special.
\newblock {\em J. Amer. Math. Soc.}, 31(2):319--347, 2018.

\bibitem{SW2015}
Michah Sageev and Daniel~T. Wise.
\newblock Cores for quasiconvex actions.
\newblock {\em Proc. Amer. Math. Soc.}, 143(7):2731--2741, 2015.

\bibitem{WiseManuscript}
Daniel~T. Wise.
\newblock The structure of groups with a quasiconvex hierarchy.
\newblock Princeton University Press, 2021.

\end{thebibliography}
\end{document}